\documentclass[11pt]{amsart}
\usepackage[colorlinks,citecolor=red, dvipdfm]{hyperref}

\newtheorem{theorem}{Theorem}[section]
\newtheorem{lemma}[theorem]{Lemma}

\theoremstyle{definition}

\newtheorem{corollary}[theorem]{Corollary}

\theoremstyle{remark}

\newcommand{\bs}{\begin{split}}
\newcommand{\es}{\begin{split}}

\newcommand{\be}{\begin{equation}}
\newcommand{\ee}{\end{equation}}

\numberwithin{equation}{section}

\begin{document}

\title[A characterization of the $\hat{A}$-genus]
{A characterization of the $\hat{A}$-genus as a linear combination
of Pontrjagin numbers}

%    Information for first author
\author{Ping Li}
%    Address of record for the research reported here
\address{Department of Mathematics, Tongji University, Shanghai 200092, China}
%\address{Department of Mathematics, Faculty of Science and Engineering, Waseda University, Shinjuku, Tokyo 169-8555, Japan}
%\curraddr{Department of Mathematics and Information Sciences, Tokyo
%Metropolitan University, Tokyo 192-0397, Japan}
%    Current address
%\curraddr{Department of Mathematics and Information Sciences, Tokyo
%Metropolitan University, Tokyo 192-0397, Japan}
\email{pingli@tongji.edu.cn}
%    \thanks will become a 1st page footnote.
\thanks{The author was partially supported by the National
Natural Science Foundation of China (Grant No. 11471247) and the
Fundamental Research Funds for the Central Universities.}

%    General info

%\date{January 1, 2001 and, in revised form, June 22, 2001.}

%\dedicatory{This paper is dedicated to our advisors.}

\begin{abstract}
We show in this short note that if a rational linear combination of
Pontrjagin numbers vanishes on all simply-connected $4k$-dimensional
closed connected and oriented spin manifolds admitting a Riemannian
metric whose Ricci curvature is nonnegative and nonzero at any
point, then this linear combination must be a multiple of the
$\hat{A}$-genus, which improves on a result of Gromov and Lawson.
Our proof combines an idea of Atiyah and Hirzebruch and the
celebrated Calabi-Yau theorem.
\end{abstract}

\maketitle
\section{Introduction and main observation}
Unless otherwise stated, all manifolds considered in this note are
smooth, closed, connected and oriented. All metrics mentioned in
this note refer to Riemannian metrics.

Atiyah and Hirzebruch proved in \cite{AH} that if a $4k$-dimensional
spin manifold $M$ admits a non-trivial (smooth) circle action
($S^1$-action), then the $\hat{A}$-genus of $M$ must vanish. Indeed,
they showed via the Atiyah-Bott-Singer fixed point formula that the
equivariant index $\text{spin}(g,M)\in\text{R}(S^1)$ of the Dirac
operator on $M$ vanishes identically. In particular, the
$\hat{A}$-genus $\hat{A}(M)=\text{spin}(1,M)$ must vanish. Note that
the $\hat{A}$-genus is a rational linear combination of Pontrjagin
numbers and $\hat{A}(\cdot)$ can be viewed as a ring homomorphism
$\hat{A}(\cdot):~\Omega^{\text{SO}}_{\ast}\otimes\mathbb{Q}\rightarrow\mathbb{Q}$,
where $\Omega^{\text{SO}}_{\ast}$ is the oriented cobordism ring
(\cite{Hi}). In addition to the above-mentioned main result, by
suitably choosing generators of
$\Omega^{\text{SO}}_{\ast}\otimes\mathbb{Q}$ and a clever
manipulation for them, Atiyah and Hirzebruch also proved in
\cite{AH} that the $\hat{A}$-genus can be characterized as the only
linear combination of Pontrjagin numbers that vanishes on all
$4k$-dimensional spin manifolds admitting a non-trivial circle
action. To be more precise, they showed that (\cite[\S 2.3]{AH}) a
rational linear combination of Pontrjagin numbers vanishes on all
spin manifolds equipped with a non-trivial smooth circle action if
and only if it is a multiple of the $\hat{A}$-genus.

Recall that on $4k$-dimensional spin manifolds the vanishing of the
$\hat{A}$-genus is also an obstruction to the existence of a
positive scalar curvature metric, which is a classical result of
Lichnerowicz (\cite{Li}) and whose argument is based on a
Bochner-type formula related to the Dirac operator and the
Atiyah-Singer index theorem (a detailed and excellent proof can be
found in \cite[Ch.5]{Wu}). Another breakthrough related to the
existence of a positive scalar curvature metric on spin manifolds
came from Gromov and Lawson. They proved in \cite[Theorem B]{GL}
that if a simply-connected spin manifold $X$ is spin cobordant to a
manifold equipped with a positive scalar curvature metric, then $X$
itself also carries
 such a metric. As an application of this result, Gromov-Lawson obtained the same type result as that of
 Atiyah-Hirzebruch in the language of cobordism theory (\cite[Corollary B]{GL}):
 a rational linear combination of Pontrjagin numbers vanishes on all $4k$-dimensional spin manifolds carrying
 a positive scalar curvature metric if and only if it is a multiple of the
 $\hat{A}$-genus.

 Another result which has a similar flavor to the above Atiyah-Hirzebruch and Gromov-Lawson's results
 should also be mentioned. In \cite{Ko}, Kotschick showed that a rational linear combination of Pontrjagin numbers
  which is bounded on all $4k$-dimensional manifolds with a nonnegative sectional curvature metric
  if and only if it is a multiple of the signature. After \cite[Corollary 2]{Ko}, the author commented that his result
  bears some resemblance to Gromov-Lawson's above result.

  However, although Atiyah-Hirzebruch's result is much earlier, clearly neither \cite{GL} nor \cite{Ko} was aware of
  it as they didn't cite \cite{AH}. In spirit, both Gromov-Lawson and Kotschick's proofs
  are similar to that of Atiyah-Hirzebruch (see \cite[p. 432]{GL} and \cite[p. 139]{Ko}):
  suitably choose generators for $\Omega^{\text{SO}}_{\ast}\otimes\mathbb{Q}$ sastisfying certain
  restrictions and then apply a proportionality argument to lead to the desired
  proof. This idea was further taken up by Kotschick et al to solve a
  long-standing problem posed by Hirzebruch in 1954
  (\cite{Ko2}, \cite{KS}).

The \emph{main purpose} of this short note is to show that, by using
the generators of $\Omega^{\text{SO}}_{\ast}\otimes\mathbb{Q}$
chosen in \cite{AH}, together with the celebrated Calabi-Yau
theorem, we can give a direct proof of Gromov-Lawson's this result,
which avoids using their main result \cite[Theorem]{GL}. Moreover,
because of the use of the Calabi-Yau theorem, our statement is
indeed more sharper and thus improves on their original result. We
now state our main observation in this note, whose proof will be
given in Section \ref{section2}.

Note that the Ricci curvature of a metric on a manifold is an
assignment of a quadratic form at the tangent space of any point of
this manifold. The Ricci curvature is said to be \emph{nonnegative
and nonzero at any point} if, at any point of this manifold, this
quadratic form is nonnegative definite but nonzero. With this notion
understood, our main observation is the following

\begin{theorem}\label{maintheorem}
If a rational linear combination of Pontrjagin numbers vanishes on
all $4k$-dimensional simply-connected spin manifolds equipped with a
metric whose Ricci curvature is nonnegative and nonzero at any
point, then it must be a multiple of the $\hat{A}$-genus.
\end{theorem}

Clearly the conclusion in Theorem \ref{maintheorem} remains true if
we weaken the restrictions imposed on the manifolds in Theorem
\ref{maintheorem}. Recall that the scalar curvature is nothing but
the trace of the real symmetric matrix which represents the
quadratic form of the Ricci curvature under some orthonormal basis.
So the condition that the Ricci curvature be nonnegative and nonzero
at any point means that the eigenvalues of this matrix are all
nonnegative and at least one of them is positive. In particular the
trace is positive. This means the condition of the Ricci curvature
required in Theorem \ref{maintheorem} implies that the scalar
curvature of this metric is positive. Therefore, Theorem
\ref{maintheorem}, together with the Lichnerowicz vanishing theorem,
yields the following corollary.

\begin{corollary}(Gromov-Lawson, \cite[Corollary B]{AH})\label{maincoro}
A rational linear combination of Pontrjagin numbers vanishes on all
$4k$-dimensional spin manifolds equipped with a positive scalar
curvature metric if and only if it is a multiple of the
$\hat{A}$-genus.
\end{corollary}

\section{Proof of Theorem \ref{maintheorem}}\label{section2}
Let $\mathbb{H}P^k$ denote the quaternionic projective space whose
real dimension is $4k$. It is well-known that $\mathbb{H}P^k$ admits
a metric with positive sectional curvature and thus it has positive
Ricci curvature. Let $V_2(4)$ be the smooth hypersurface of degree
$4$ in the $3$-dimensional complex projective space $\mathbb{C}P^3$,
which is a compact K\"{a}hler surface and whose first Chern class
$c_1(V_2(4))=0$ in $H^2(V_2(4);\mathbb{Z})$ (Kummer surface). For
simplicity we denote by $N^1:=V_2(4)$ and $N^k:=\mathbb{H}P^k$
($k\geq 2$).

For our later use, we record some facts related to $N^k$ in the
following lemma. Although some of them have been sketchily explained
in \cite[\S 2.3]{AH}, for the reader's convenience, we would like to
either point out a reference or indicate its proof if some
non-standard fact is stated/recorded.

\begin{lemma}\label{lemma}
~\begin{enumerate}
\item
$$\Omega^{\text{SO}}_{\ast}\otimes\mathbb{Q}=
\mathbb{Q}[N^1,N^2,N^3,\ldots,].$$ In other words,
$\{N^k\}^{\infty}_{k=1}$ is a basis sequence for the oriented
cobordism graded ring tensored with $\mathbb{Q}$. This means, if we
denote by $\Omega^k$ the restriction of $\Omega^{\text{SO}}_{\ast}$
to the $4k$-dimensional manifolds, then the vector space
$\Omega^k\otimes\mathbb{Q}$ has a basis $\{\prod_{i=1}^tN^{k_i}\}$,
where $(k_1,\ldots,k_t)$ runs over all the partitions of weight $k$
$(=k_1+k_2+\cdots +k_t).$

\item
All $N^k$ and any of their finitely many product
$\prod_{i=1}^tN^{k_i}$ are simply-connected and spin.

\item
$\hat{A}(N^1)=2$ and $\hat{A}(N^k)=0$ with $k\geq 2$.

\item
Suppose we have a sequence of $t$ positive-dimensional Riemannian
manifolds $(M_i,g_i)$ $(1\leq i\leq t)$ such that any Ricci
curvature $\text{Ric}(g_i)$ is either positive or identically zero
and at least one of these $\text{Ric}(g_i)$ is positive. Then the
Ricci curvature $\text{Ric}(\prod_{i=1}^tg_i)$ of the Riemannian
product manifold $(\prod_{i=1}^tM_i,\prod_{i=1}^tg_i)$ is
nonnegative and nonzero at any point.

\item
The product manifold $\prod_{i=1}^tN^{k_i}$ admits a Riemanian
metric whose Ricci curvature is nonnegative and nonzero at any point
if at least one of these $k_i$ is larger than $1$.
\end{enumerate}
\end{lemma}
\begin{proof}~
\begin{enumerate}
\item
This was observed in \cite[\S 2.3]{AH}. Indeed, there exists a
simple procedure, which is due to Thom, to detect whether or not a
given sequence of $4k$-dimensional manifolds $(k=1,2,\ldots)$ is a
basis sequence (cf. \cite[p. 79]{Hi} or \cite[p. 43]{HBJ}).

\item
Simply-connectedness is clear. A manifold is spin if and only if its
second Stiefel-Whitney class $\omega_2=0$. $N^k$ $(k\geq 2)$ are
spin as $H^2(N^k;\mathbb{Z}_2)=0$. $N^1$ is spin as $c_1(N^1)=0,$
whose module $2$ reduction is exactly $\omega_2(N^1)$. The fact that
the product $M_1\times M_2$ of two (orientable) spin manifolds $M_1$
and $M_2$ is still spin follows from the Whitney direct sum formula
for
 $\omega_2(M_1\times M_2)$ and the fact that the first
Stiefel-Whitney class is zero if and only if the manifold is
orientable (\cite{MS}).

\item
$\hat{A}(N^k)=0$ for $k\geq 2$ follow form the main result in
\cite{AH} as these $N^k$ admit nontrivial circle actions.
$\hat{A}(N^1)=2$ is quite well-known. In fact there is a closed
formula to calculate the $\hat{A}$-genus of general complete
intersections in the complex projective spaces (\cite{Br}).

\item
This statement follows from the simple fact that the quadric form of
the Ricci curvature $\text{Ric}(g_1\times g_2)$ of the product
metric $g_1\times g_2$ is exactly the direct sum of those of
$\text{Ric}(g_1)$ and $\text{Ric}(g_2)$.

\item
As we have mentioned that on each $N^k$ ($k\geq 2$) one has a
positive Ricci curvature metric. Since $c_1(N^1)=0$, the celebrated
Calabi-Yau theorem, which refers to S.-T. Yau's solution to the
famous Calabi conjecture (\cite{yau}), tells us that $N^1$ carries a
K\"{a}hler (hence Riemannian) metric whose Ricci curvature is
identically zero. Hence this statement follows from $(4)$. Note that
it is \emph{this} place where we apply the remarkable Calabi-Yau
theorem!
\end{enumerate}
\end{proof}

Now we are in a position to prove Theorem \ref{maintheorem}.

\begin{proof}
Fix a positive integer $k$.

Let $\hat{A}_k$ be the restriction of $\hat{A}$ to
$\Omega^k\otimes\mathbb{Q}$. Namely, $\hat{A}_k$ is a linear
homomorphism:
$$\hat{A}_k(\cdot):~\Omega^k\otimes\mathbb{Q}\rightarrow\mathbb{Q}.$$
Let $q$ be a rational linear combination of Pontrjagin numbers on
$4k$-dimensional manifolds. Since Pontrjagin numbers are cobordism
invariants, $q$ can also be viewed as a linear homomorphism
$$q:~\Omega^k\otimes\mathbb{Q}\rightarrow\mathbb{Q}.$$
 In view of these, it suffices to
show that, if $q$ vanishes on all spin manifolds which admits a
metric whose Ricci curvature is nonnegative and nonzero at any
point, then $q$ is a multiple of $\hat{A}_k$.

If $k=1$, the conclusion is clear as there has only one Pontrjagin
number $p_1$ and $\hat{A}_1=\frac{1}{24}p_1$. We now suppose that
$k\geq 2$. By $(1)$ of Lemma \ref{lemma} we know that the vector
space $\Omega^k\otimes\mathbb{Q}$ has a basis
$\{\prod_{i=1}^tN^{k_i}\}$, where $(k_1,\ldots,k_t)$ runs over all
the partitions of weight $k$. $(5)$ of Lemma \ref{lemma} tells us
that $q$ vanishes on all these basis elements with only one possible
exception $(N^1)^k$. The same holds for $\hat{A}_k$ by $(3)$ of
Lemma \ref{lemma}. This implies
$$q-\frac{q\big((N^1)^k\big)}{2^k}\hat{A}_k$$
 vanishes on all these
basis elements and thus
$$q=\frac{q\big((N^1)^k\big)}{2^k}\hat{A}_k~\text{
on}~\Omega^k\otimes\mathbb{Q},$$
 which gives the desired proof.
\end{proof}


\begin{thebibliography}{10}
\bibitem[AH70]{AH}
{M.F. Atiyah, F. Hirzebruch}:
\newblock {\em Spin-Manifolds and Group Actions},
\newblock Essays on Topology and Related Topics, M\'{e}moires d\'{e}di\'{e}s \`{a} Georges de Rham, A. Haefliger and R. Narasimhan, ed., Springer-Verlag, New York-Berlin, 1970, pp. 18-28.

\bibitem[Br83]{Br}
{R. Brooks}:
\newblock {\em The $\hat{A}$-genus of complex hypersurfaces and complete intersections},
\newblock Proc. Amer. Math. Soc. {\bf 87} (1983), 528-532.


\bibitem[GL80]{GL}
{M. Gromov, H.B. Lawson, Jr.}:
\newblock {\em The classification of simply connected manifolds of positive scalar curvature},
\newblock Ann. Math. {\bf 111} (1980), 423-434.

\bibitem[Hi66]{Hi}
{F. Hirzebruch}:
\newblock {\em Topological methods in algebraic geometry},
\newblock 3rd Edition, Springer, Berlin (1966).

\bibitem[HBJ92]{HBJ}
{F. Hirzebruch, T. Berger, R. Jung}:
\newblock {\em Manifolds and modular forms},
\newblock Aspects of Mathematics, E20, Friedr. Vieweg and Sohn, Braunschweig, 1992.

\bibitem[Ko10]{Ko}
{D. Kotschick}:
\newblock {\em Pontryagin numbers and nonnegative curvature},
\newblock  J. Reine Angew. Math. {\bf 646} (2010), 135-140.

\bibitem[Ko12]{Ko2}
{D. Kotschick}:
\newblock {\em Topologically invariant Chern numbers of projective varieties},
\newblock  Adv. Math. {\bf 229} (2012), 1300-1312.

\bibitem[KS13]{KS}
{D. Kotschick, S. Schreieder}:
\newblock {\em The Hodge ring of K\"{a}hler manifolds},
\newblock  Compos. Math. {\bf 149} (2013), 637-657.

\bibitem[Li63]{Li}
{A. Lichnerowicz}:
\newblock {\em Spineurs harmoniques},
\newblock  C. R. Acad. Sci. Paris {\bf 257} (1963), 7-9.

\bibitem[MS74]{MS}
{J.W. Milnor, J.D. Stasheff}:
\newblock {\em Characteristic classes},
\newblock Princeton University Press, Princeton, N. J., 1974. Annals of Mathematics Studies, No. 76.

\bibitem[Wu88]{Wu}
{H.-H. Wu}:
\newblock {\em The Bochner technique in differential geometry},
\newblock  Math. Rep. 3 (1988), no. 2, i-xii and 289-538.


\bibitem[Yau77]{yau}
{S.-T. Yau}:
\newblock {\em Calabi's conjecture and some new results in algebraic geometry},
\newblock  Proc. Natl. Acad. Sci. USA, {\bf 74} (1977), 1798-1799.
\end{thebibliography}
\end{document}